\documentclass[10pt]{amsart}

\usepackage{graphicx, xcolor}
\usepackage{amssymb}
\usepackage{bbm}
\usepackage{amsmath,mathtools}
\usepackage{amsthm}
\usepackage{stmaryrd}
\usepackage{tikz-cd}
\usepackage{faktor}
\usepackage{enumerate}
\usepackage{hyperref}

\newcommand{\G}{\mathcal{G}}
\newcommand{\s}{\mathfrak{s}}
 
\newcommand{\R}{\mathbb {R}} 
\newcommand{\Q}{\mathbb {Q}} 
\newcommand{\Z}{\mathbb {Z}} 
\newcommand{\N}{\mathbb {N}} 
\newcommand{\D}{\mathbb {Z}[k^{-1}]}
\renewcommand{\O}{\mathcal{O}}

\newcommand{\K}{\mathcal{K}} 
\newcommand{\tr}{\mathrm{tr}} 
\newcommand{\Tr}{\mathrm{Tr}} 
\newcommand{\Ker}{\mathrm{Ker}} 
\newcommand{\Cu}{\mathrm{Cu}}
\newcommand{\Prim}{\mathrm{Prim}}
\renewcommand{\Im}{\mathrm{Im}}
\newcommand{\id}{\mathrm{id}}

\theoremstyle{definition}

\newtheorem{thm}{Theorem}[section]
\newtheorem{lem}[thm]{Lemma}
\newtheorem{prp}[thm]{Proposition}
\newtheorem{dfn}[thm]{Definition}
\newtheorem{cor}[thm]{Corollary}
\newtheorem{rmk}[thm]{Remark}

\newtheorem*{dfn*}{Definition}

\numberwithin{equation}{section}

\title{
C*-diagonals with Cantor spectrum in Cuntz algebras
}

\author{Samuel Evington}
\address{Samuel Evington, Mathematical Institute, University of M\"unster, Ein\-stein\-strasse 62, 48149 M\"unster, Germany}
\email{evington@uni-muenster.de}

\author{Philipp Sibbel}
\address{Philipp Sibbel, Mathematical Institute, University of M\"unster, Ein\-stein\-strasse 62, 48149 M\"unster, Germany}
\email{philipp.sibbel@uni-muenster.de}

\thanks{Funded by the Deutsche Forschungsgemeinschaft (DFG, German Research Foundation) under Germany’s Excellence Strategy – EXC 2044 – 390685587, Mathematics Münster – Dynamics – Geometry – Structure; the Deutsche Forschungsgemeinschaft (DFG, German Research Foundation) – Project-ID 427320536 – SFB 1442; and ERC Advanced Grant 834267 - AMAREC}

\begin{document}
\begin{abstract}
We prove that there exists a C$^*$-diagonal with Cantor spectrum in the Cuntz algebra $\mathcal{O}_k$ for $2 \leq k < \infty$. 
Our method generalises to an uncountable family of UCT Kirchberg algebras with distinct K-theory.
Moreover, we construct principal étale groupoid models for these Cuntz algebras and UCT Kirchberg algebras. 
\end{abstract}

\maketitle
\allowdisplaybreaks
\section*{Introduction}
\numberwithin{equation}{section}
\renewcommand{\thethm}{\Alph{thm}}

The theory of Cartan subalgebras of C$^*$-algebras was initiated by Kumjian (\cite{kumjian1986c}) and Renault (\cite{renault2008cartan}), building on the work of Vershik (\cite{Ver71}) and Feldman--Moore (\cite{FM77}) in the von Neumann setting. Every Cartan subalgebra $D \subseteq A$ corresponds to a twisted étale groupoid model for the C$^*$-algebra $A$. For this reason, the question of the existence and (non)uniqueness of Cartan subalgebras has received much attention (\cite{barlak2017cartan, li2019cartan, li2020every, li2022constructing}) aided by the tools developed in the Elliott classification programme (\cite{El95,Ki95,Ph00, winter2018structure, WhiteICM}).

This paper is concerned with the special class of Cartan subalgebras that correspond to twisted equivalence relations, i.e.\ principal groupoids. 
Kumjian called these subalgebras C$^*$-diagonals (\cite{kumjian1986c}). C$^*$-diagonals arise naturally in the study of regularity properties for Cartan pairs (\cite{KW24}). In particular, finite diagonal dimension forces a Cartan subalgebra to be a C$^*$-diagonal (\cite{LLW23}). We therefore expect C$^*$-diagonals to play a significant role in any future K-theoretic classification of Cartan subalgebras. 

The motivating example of a C$^*$-diagonal is the subalgebra of diagonal matrices $D_k \subseteq M_k$. 
Taking infinite tensor products, one obtains a C$^*$-diagonal $D_{k^\infty}$ with Cantor spectrum in the uniformly hyperfinite (UHF) C$^*$-algebra $M_{k^\infty}$. 
The Cuntz algebra $\O_k$ contains a canonical copy of $M_{k^\infty}$ and thus $D_{k^\infty}$ for $2 \leq k < \infty$ (\cite{cuntz1977simple}); 
however, the inclusion $D_{k^\infty} \subseteq \O_k$ is only a Cartan subalgebra, not a C$^*$-diagonal. 
Indeed, Renault constructed a groupoid model for this inclusion that is topologically principal but not principal (\cite{renault1980groupoid}).
More generally, several authors have shown that every Kirchberg algebra satisfying the universal coefficient theorem (UCT) has a Cartan subalgebra by constructing topologically principal groupoid models (see \cite{spielberg2007graph,katsura2008class,Yeend07,li2019cartan, CFH20}), but the constructions typically lead to non-principal groupoids, so give rise to Cartan subalgebras that are not C$^*$-diagonals. 
It is therefore an open question which UCT Kirchberg algebras contain a C$^*$-diagonal.

This question has attracted considerable attention over the years. Hjelmborg proved that $\O_2$ does in fact contain a C$^*$-diagonal (\cite{hjelmborgstability}). R{\o}rdam and Sierakowski showed that purely infinite C$^*$-algebras can arise as crossed products when a non-amenable group acts amenably on a topological space (\cite{rordam2012purely}), which leads to some further examples of C$^*$-diagonals in UCT Kirchberg algebras (see \cite{elliott2016k, suzuki2017amenable}). 
Moreover, Brown, Clark, Sierakowski, and Sims proved that a C$^*$-diagonal exists in UCT Kirchberg algebras under certain K-theoretic conditions (\cite{brown2016purely}). 

Recently, Winter and the second named author proved that $\O_2$ contains a C$^*$-diagonal with Cantor spectrum, and asked if this result generalises to $\O_k$ (\cite{sibbel2024cantor}). The main result of this paper is the following. 
\begin{thm}\label{thm:Diagonal-in-Ok}
   The Cuntz algebra $\O_k$ contains a C$^*$-diagonal with Cantor spectrum for every $2 \leq k < \infty$.
\end{thm}

By the general theory of C$^*$-diagonals (\cite{kumjian1986c, renault2008cartan}), it follows from Theorem \ref{thm:Diagonal-in-Ok} that there is a \emph{twisted} étale groupoid model for $\O_k$ that is principal. However, using the theory of topological graphs (\cite{katsura2004class}), we can show that no twist is required.
\begin{thm}\label{thm:Groupoid-model-for-Ok}
   For every $2 \leq k < \infty$, there exists a second countable, principal, étale groupoid $\mathcal{G}_k$ whose unit space is a Cantor space such that $C^*_r(\mathcal{G}_k) = \O_k$.
\end{thm}
Moreover, we give a concrete description of the groupoid $\mathcal{G}_k$ as the product of the Deaconu--Renault groupoid associated to a topological graph and the groupoid of an AF equivalence relation.

The case $k=2$ of Theorem \ref{thm:Diagonal-in-Ok} was proven by Winter and the second named author (\cite{sibbel2024cantor}). Their idea was to transform the canonical Cartan subalgebra $D_{2^\infty} \subseteq \O_2$ to a C$^*$-diagonal with Cantor spectrum in a larger C$^*$-algebra $B_{2,\alpha}$, with the aid of a Cantor minimal system $\alpha$, and then prove that $\bigotimes_{n\in\N} B_{2,\alpha} \cong \O_2$ using Kirchberg's celebrated $\O_2$-absorption theorem (\cite{Ki95}). This paper provides an alternative proof, based on K-theoretic computations and the Kirchberg--Phillips theorem (\cite{Ki95,Ph00}), which can be generalised to handle $\mathcal{O}_k$ for all $2 \le k < \infty$.  

This project was motivated by a question asked by Mikael R{\o}rdam at the conference ``Noncommutativity in the North'' in Gothenburg in the Summer of 2024. He asked whether  $B_{2,\alpha}$ is already isomorphic to $\O_2$. We answer this question negatively in this paper by showing that $K_0(B_{2,\alpha}) \neq 0$. In fact, this $K_0$-group depends on the Cantor minimal system $\alpha$. 
When $\alpha$ is an odometer, the $K_0$-group we obtain is an increasing union of finite cyclic groups, whose sizes depend on the choice of odometer.
After generalising these arguments to the $\O_k$-setting, we obtain the following theorem, which gives uncountably many new examples of UCT Kirchberg algebras that contain a C$^*$-diagonal with Cantor spectrum.

\begin{thm}\label{thm:Diagonal-in-Kirchberg}
    Let $k \geq 2$ and let $(n_i)_{i\in \N}$ be a strictly increasing sequence in $\N$ such that $n_i \mid n_{i+1}$. Set $m_i = k^{n_i}-1$ for $i \in\N$. 
    Then there exists a unital Kirchberg algebra $B$ satisfying the UCT with a C$^*$-diagonal $D \subseteq B$ whose spectrum is the Cantor space, and the K-theory of $B$ is given by   
	\begin{align*}
		K_0(B) &= \varinjlim \left(\Z_{m_i},\eta_{i,i+1}\right)\\
		K_1(B) &= 0,
	\end{align*}
    where $\eta_{i,i+1}(x) = \tfrac{m_{i+1}}{m_i} x$ for $x \in \Z_{m_i}$.
\end{thm}
 
We deduce Theorem~\ref{thm:Diagonal-in-Ok} from Theorem \ref{thm:Diagonal-in-Kirchberg} by taking the tensor product with a suitable UHF algebra in order to annihilate the undesired torsion elements. This results in a UCT Kirchberg algebra that is isomorphic to $\O_k$ by the Kirchberg--Phillips theorem. Since UHF algebras contain a canonical C$^*$-diagonal with Cantor spectrum this completes the proof. Theorem \ref{thm:Groupoid-model-for-Ok} is deduced by realising $B_{2,\alpha}$ and its generalisations as topological graph C$^*$-algebras.

\subsection*{Structure of paper}
Section \ref{sec:preliminaries} sets up notation and recalls the necessary preliminary results. 
In Section \ref{section:Construction}, we generalise the first step of the construction from \cite{sibbel2024cantor} to the $\O_k$-setting and show that the resulting C$^*$-algebras $B_{k,\alpha}$ are unital UCT Kirchberg algebras. 
In Section \ref{sec:K-theory}, we compute the K-theory of $B_{k,\alpha}$, answer R{\o}rdam's question and prove Theorem \ref{thm:Diagonal-in-Kirchberg}. 
In Section \ref{sec:main-theorem}, we deduce Theorem \ref{thm:Diagonal-in-Ok}. In Section \ref{sec:groupoid-model}, we prove Theorem \ref{thm:Groupoid-model-for-Ok}.

\section{Preliminaries}\label{sec:preliminaries}
\renewcommand{\thethm}{\arabic{thm}}
\numberwithin{thm}{section}

\subsection{Notation}\label{subsec:notation}
Let $k \in \N$. We write $\Z_k = \Z / k\Z$ for the cyclic group of order $k$. An element $x + k\Z \in \Z_k$ will be abbreviated to $x$, provided there is no risk of confusion. We write $\Z[k^{-1}]$ for the subring of $\Q$ generated by $\Z$ and $k^{-1}$.
We write $M_k$ for the $k \times k$ matrices and $\K$ for the C$^*$-algebra of compact operators on $\ell^2(\N)$. We reserve the notation $e_{ij}$ for the standard basis elements in $M_k$ and $\K$.
 
Let $A$ be a C$^*$-algebra. We write $A_+$ for the positive elements in $A$. We denote Murray--von Neumann subequivalence of projections $p,q \in A$ by $p \precsim q$, and we use the same notation for Cuntz subequivalence of positive elements, since these notions of subequivalence coincide for projections; see for example \cite[Lemma 2.8]{GardellaPerera}.
We write $V(A)$ for the Murray--von Neumann semigroup of $A$ and $\Cu(A)$ for the Cuntz semigroup of $A$.

Let $A$ be a C$^*$-algebra and $X$ be a compact Hausdorff space. We write $C(X,A)$ for the C$^*$-algebra of continuous functions $X \rightarrow A$, which we identify with the tensor product 
$A \otimes C(X)$. 
Let $G$ be an ordered group. We write $C_{lc}(X,G)$ for the ordered group of locally constant functions $X \rightarrow G$ with pointwise order.

We write $M_\s$ for the UHF algebra with supernatural number $\s$ and use the abbreviation $\K M_\s = \K \otimes M_\s$ for its stabilisation. We write $\tr_\s$ for the unique tracial state on $M_\s$ and set $\Tr_\s = \Tr_\K \otimes \tr_\s$. 

For $2 \leq k < \infty$, we recall that the Cuntz algebra $\O_k$ is the universal C$^*$-algebra generated by isometries $S_1,\ldots,S_k$ such that $\sum_{i=1}^k S_iS_i^* = 1_{\O_k}$. We write $S_\mu = S_{\mu_1}S_{\mu_2}\cdots S_{\mu_n}$ for $\mu \in \{1,2,\ldots,k\}^n$.
In \cite[Section 2.1]{cuntz1977simple}, Cuntz showed that the stabilised Cuntz algebra $\K \otimes \O_k$ is isomorphic to the crossed product $\K M_{k^\infty} \rtimes_{\Phi_k} \Z$ for $2 \leq k < \infty$, where $\Phi_k$ is an automorphism of $\K M_{k^\infty}$ that scales the trace by $\frac{1}{k}$. We shall refer to $\Phi_k$ as \emph{the trace scaling automorphism} on $\K M_{k^\infty}$. It is uniquely determined up to approximate unitary equivalence by the fact that it scales the trace by $\frac{1}{k}$. 

For completeness, we recall the construction of the trace scaling automorphism.
Let $\phi_k : \K \to \K \otimes M_k$ be an isomorphism with $\phi_k(e_{11}) = e_{11} \otimes e_{11}$.
Since $M_{k^\infty} \cong \bigotimes_{n \in \N} M_k$, there is a canonical isomorphism $\psi_k: M_k \otimes M_{k^\infty} \rightarrow M_{k^\infty}$ defined by relabelling the tensor factors.
Let $\Phi_k : \K M_{k^\infty} \to \K M_{k^\infty}$ be the composition
\begin{equation}\label{eqn:defPhi}
    \mathcal{K} \otimes M_{k^\infty}  \xrightarrow{\phi_k \otimes \id_{M_{k^\infty}}} \mathcal{K} \otimes M_k \otimes M_{k^\infty} \xrightarrow{\id_\K \otimes \psi_k} \K \otimes M_{k^\infty}.
\end{equation}
Then $\Tr_{k^\infty}(\Phi_k(x)) = \tfrac{1}{k}\Tr_{k^\infty}(x)$ for all elements $x \in \K M_{k^\infty}$ of finite trace.

\subsection{Cartan subalgebras and \texorpdfstring{C$^*$}{C*}-diagonals}
\label{subsec:Diagonals}

Both Cartan subalgebras and C$^*$-diagonals are natural classes of maximal abelian subalgebras that have close connections to groupoid C$^*$-algebras. In this subsection,  we recall the definitions, and in the next subsection we shall review the connection to groupoid C$^*$-algebras. 

\begin{dfn}[{cf.\ \cite[Definition 5.1]{renault2008cartan}}]
Let $A$ be a C$^*$-algebra. Then $D \subseteq A$ is called a \emph{Cartan subalgebra} if 
\begin{enumerate}
\item[(0)] $D$ contains an approximate unit of $A$,
\item[(1)] $D$ is a maximal abelian sub-C$^*$-algebra of $A$,
\item[(2)] $A$ is generated by  $\mathcal{N}_A(D) = \{n \in A : n^*Dn \subseteq D \text{ and } nDn^* \subseteq D\}$,
\item[(3)] there exists a faithful conditional expectation of $A \rightarrow D$.
\end{enumerate}
Moreover, $D \subseteq A$ is called a C$^*$-\textit{diagonal} if additionally
\begin{enumerate}
\item[(4)] every pure state on $D$ has a unique extension to a pure state on $A$.
\end{enumerate}
\end{dfn}
Condition (0) is easily seen to be redundant in the unital case, as $1_A \in D$ by maximality. In fact, it is now known to be redundant in general (\cite{pitts2021normalizers}), so can be safely removed from the definition. The original definition of C$^*$-diagonals is due to Kumjian (\cite{kumjian1986c}) and can be shown to be equivalent to the one given above.

\subsection{Étale Groupoids}
\label{subsec:Groupoids}

We assume the reader has a basic understanding of topological groupoids (see for example \cite{renault1980groupoid, renault2008cartan}). Given a topological groupoid $\G$, we write $\G^{(0)}$ for the unit space of $\G$ and $r,s:\G \rightarrow \G^{(0)}$ for the range and source maps. All topological groupoids considered in this paper will be locally compact, Hausdorff, and étale. For brevity, we shall refer to this class of groupoids as \emph{étale groupoids} with the other hypotheses implicit. 

For each $x \in \G^{(0)}$, the group $\G_x^x = \{g \in \G: r(g) = s(g) = x\}$ is called the \emph{isotropy subgroup} of $\G$ at $x$.
The groupoid $\G$ is said to be \emph{topologically principal} if the set of all $x \in \G^{(0)}$ with $\G_x^x = \{x\}$ is dense in $\G^{(0)}$, and is said to be \emph{principal} if  $\G_x^x = \{x\}$ for all $x \in \G^{(0)}$.

We write $C_r^*(\G,\Sigma)$ for the reduced C$^*$-algebra of the étale groupoid $\G$ with twist $\Sigma$ (see \cite[Section 4]{renault2008cartan} for details of the construction).
Building on work of Kumjian (\cite{kumjian1986c}), Renault established a 1-to-1 correspondence between Cartan subalgebras and twisted étale groupoids. The following theorem summarises \cite[Theorems 5.2 and 5.9]{renault2008cartan} and \cite[Proposition 5.11]{renault2008cartan}.

\begin{thm}[{\cite{renault2008cartan}}]
    Let $\G$ be a second countable, topologically principal, étale groupoid and $\Sigma$ be a twist. Then $C_0(\G^{(0)}) \subseteq C_r^*(\G,\Sigma)$ is a Cartan subalgebra. Conversely, up to isomorphism, every Cartan subalgebra $D \subseteq A$ with $A$ separable is of this form, and the twisted étale groupoid $(\G, \Sigma)$ is uniquely determined up to isomorphism. Moreover, under this 1-1 correspondence, $D \subseteq A$ is a C$^*$-diagonal if and only if $\G$ is principal.          
\end{thm}

A useful corollary of this 1-1 correspondence is that the minimal tensor product of Cartan subalgebras (of separable C$^*$-algebras) is again a Cartan subalgebra by considering the product groupoid
(see \cite[Lemma 5.1]{barlak2017cartan}). As the product of principal groupoids is again principal, we see that the minimal tensor product of C$^*$-diagonals (of separable C$^*$-algebras) is again a C$^*$-diagonal.  

\subsection{K-theoretic preliminaries}
Let $\s$ be a supernatural number. We recall that Murray--von Neumann subequivalence in $\K M_\s$ is determined by $\Tr_\s$ and 
\begin{align}\label{eqn:K0-UHF}
    K_0(M_\s) \cong \left\{\frac{a}{b}: a, b \in \Z, b | \s \right\} \subseteq \Q
\end{align}
as ordered groups. In this paper, we shall treat \eqref{eqn:K0-UHF} as an identification using the isomorphism induced by $\Tr_\s$. 
We make the following observation about quotients of $K_0(M_\s)$. 
\begin{prp}\label{prop:K0-UHF-quotient}
    Let $M_\s$ be the UHF algebra with supernatural number $\s$.
    Let $m \in \N$ be coprime to $\s$.
    Then there is an isomorphism 
    \begin{equation}
        \frac{K_0(M_\s)}{mK_0(M_\s)} \cong \Z_m 
    \end{equation}
    given by $a/b + mK_0(M_\s) \mapsto ab^{-1}$ where $b^{-1}$ denotes the multiplicative inverse of $b$ in the ring $\Z_m$. 
\end{prp}
\begin{proof}
    Let $b \in \Z$ with  $b | \s$. Since $m$ is coprime to $\s$, it follows that $b$ is coprime to $m$. Hence, $b$ has a multiplicative inverse $b^{-1}$ in $\Z_{m}$. So, there is a well-defined group homomorphism $\theta:K_0(M_\s) \rightarrow \Z_{m}$ given by $a/b \mapsto ab^{-1}$. Since $b^{-1}$ has multiplicative inverse $b$ in the ring $\Z_m$, we see that the kernel of $\theta$ is $m K_0(M_\s)$. The result now follows by the first isomorphism theorem.  
\end{proof}

In the following proposition, we collect some basic results about the algebras $C(X,\K M_\s)$ where $X$ is a zero-dimensional compact metric space. Since these results are likely well-known to experts, we provide only a sketch of the proof. Recall from Section~\ref{subsec:notation} that $C_{lc}(X, K_0(M_\s))$ is the ordered group of locally constant functions $X \rightarrow K_0(M_\s)$ with the pointwise order.\footnote{If we endow $K_0(M_\s)$ with the discrete topology, then we could write $C(X, K_0(M_\s))$ in place of $C_{lc}(X, K_0(M_\s))$. This would be justified as this is the topology on the $K_0$-group induced by the norm topology on the C$^*$-algebra. However, we have chosen to use the more explicit notation in case the reader is tempted to endow $K_0(M_\s)$ with the topology induced by $\R$.}
\begin{prp}\label{prop:K-theory-functions-into-UHF}
    Let $X$ be a zero-dimensional compact metric space. Let $\K M_\s$ be the stabilised UHF algebra with supernatural number $\s$. 
    Then $C(X, \K M_{\s})$ is a stable AF algebra with 
    \begin{align}
        V(C(X, \K M_{\s})) &\cong C_{lc}(X, K_0(M_\s))_+, \label{eqn:premlin-computation-V}\\
        K_0(C(X, \K M_{\s})) &\cong C_{lc}(X, K_0(M_\s)), \label{eqn:premlin-computation-K-0}\\
        K_1(C(X, \K M_{\s})) &= 0. \label{eqn:premlin-computation-K-1}
    \end{align}
    Moreover, in the case of $V$ and $K_0$, the isomorphism is induced by the composition with $\Tr_\s$ and is an isomorphism of ordered groups.
\end{prp}
\begin{proof}
    As $X$ is a zero-dimensional compact metric space, $C(X)$ is an AF algebra. Hence, $C(X, \K M_{\s}) \cong C(X) \otimes M_{\s} \otimes \K$ is a stable AF algebra.
    Since, $C(X, \K M_{\s})$ is AF, \eqref{eqn:premlin-computation-K-1} is immediate. Since $C(X, \K M_{\s})$ is stable to show \eqref{eqn:premlin-computation-V} and \eqref{eqn:premlin-computation-K-0} it suffices to consider projections in $C(X, \K M_{\s})$.

    Let $p \in C(X, \K M_{\s})$ be a projection. As $X$ is zero-dimensional and compact, a standard partition of unity argument shows that there exists a locally constant projection-valued function $p_0:X \rightarrow \K M_{\s}$ with $\|p-p_0\| < \frac{1}{2}$. By \cite[Proposition 4.6.6]{Bla86}, $p$ and $p_0$ are Murray--von Neumann equivalent. Hence, it suffices to consider locally constant projection-valued function. Since Murray--von Neumann equivalence in $\K M_{\s}$ is determined by $\Tr_\s$, we can deduce \eqref{eqn:premlin-computation-V} and \eqref{eqn:premlin-computation-K-0} by pointwise arguments.     
\end{proof}

One of the main technical tools in this paper is the Künneth theorem for K-theory. We shall only need the following special case.  
\begin{prp}\label{prop:UHF-Kunneth}
    Let $A$ be a C$^*$-algebra and $M_\s$ be a UHF algebra. Then 
    \begin{align}
         K_0(A \otimes M_\s) &\cong K_0(A) \otimes K_0(M_\s),\label{eqn:UHF-Kunneth-K0}\\
         K_1(A \otimes M_\s) &\cong K_1(A) \otimes K_0(M_\s).
    \end{align}
    Moreover, the isomorphism \eqref{eqn:UHF-Kunneth-K0} maps $[p \otimes q]$ to $[p] \otimes [q]$ for all projections $p \in A$ and $q \in M_\s$.
\end{prp}
\begin{proof}
    Since $K_0(M_\s)$ is torsion free, we have $\mathrm{Tor}(\, \cdot \,, K_0(M_\s)) = 0$; see for example  \cite[Corollary 3.50]{Rotman09}.
    Also, $K_1(M_\s) = 0$. 
    The result now follows by the Künneth theorem; see \cite{schochet1982topological}. 
\end{proof}

The other main technical tool is the Kirchberg--Phillips theorem (\cite{Ki95,Ph00}). We recall that a Kirchberg algebra is a simple, separable, nuclear, purely infinite C$^*$-algebra.
The following formulation of the Kirchberg--Phillips theorem can be found in \cite[Theorem 8.4.1]{rordam2002classification}. 
\begin{thm}[Kirchberg--Phillips, unital case]\label{thm:KP}
    Let $A$ and $B$ be unital Kirchberg algebras satisfying the universal coefficient theorem. Let $\alpha_0:K_0(A) \rightarrow K_0(B)$ and $\alpha_1:K_1(A) \rightarrow K_1(B)$ be group isomorphisms such that $\alpha_0([1_A]) = [1_B]$. Then there exists an isomorphism $\alpha:A \rightarrow B$ such that $K_i(\alpha)=\alpha_i$.
\end{thm}

Finally, we record the K-theory of the Cuntz algebras $\O_k$. 
\begin{thm}[\cite{Cu81}]\label{thm:Cuntz}
    Let $2 \leq k < \infty$. Then $K_0(\O_k) \cong \Z_{k-1}$, with $[1_{\O_k}]$ mapping to $1 \in \Z_{k-1}$, and  $K_1(\O_k) = 0$.
\end{thm}

\section{The main construction} \label{section:Construction} 

In this section, we revisit the main construction from \cite{sibbel2024cantor}, which in turn builds on \cite{cuntz1977simple, Cuntz1981ktheoryII}. In \cite{sibbel2024cantor}, the authors work exclusively with $\O_2$. In order to obtain C$^*$-diagonals in $\O_k$ for $2 < k <\infty$, we will need to work with a slightly more general version of the construction using $k$ isometries instead of just two.  

Let $k \in \N$ with $k \ge 2$. Let $S_1,\ldots,S_k$ be the canonical generators of the Cuntz algebra $\O_{k}$. 
Write  $D_{k^\infty}$ for the Cartan subalgebra in the Cuntz algebra $\O_{k}$, generated by elements of the form $S_\mu S_\mu^*$ for $\mu\in \{1,\ldots,k\}^n$ for all $n \in \N$. 

Let $X$ be the Cantor space and let $\alpha$ be a minimal homeomorphism on $X$. By abuse of notation, we also write $\alpha$ for the induced $^*$-automorphism of $C(X)$, given by $\alpha(f) = f \circ \alpha^{-1}$ for all $f \in C(X)$, and the $\Z$-action on $C(X)$ it generates.
Let $v \in C(X) \rtimes_\alpha \Z$ be the unitary implementing the action, i.e.\ $v f v^* = \alpha(f)$ for all $f \in C(X)$.

We now define subalgebras  $D_{k, \alpha} \subseteq B_{k, \alpha} \subseteq\mathcal{O}_k \otimes (C(X) \rtimes_\alpha \Z)$ as follows:
\begin{equation}
\begin{split}\label{eqn:def-B-k-alpha}
D_{k, \alpha} &= D_{k^\infty} \otimes C(X),\\
B_{k, \alpha} &= \textnormal{C}^*(D_{k, \alpha} \cup \{ S_1 \otimes v, \dots, S_k \otimes v\}). 
\end{split}
\end{equation}
The reason for considering this construction is the following result. 
\begin{prp}\label{prop:Diagonal-in-B}
    Let $k \geq 2$ and let $\alpha$ be a minimal homeomorphism of the Cantor space. Then $D_{k, \alpha}$ is a C$^*$-diagonal in $B_{k, \alpha}$.
\end{prp} 
\begin{proof}
    The proof of \cite[Lemma 1]{sibbel2024cantor} also works for general $k$. 
\end{proof}

The following proposition summarises the properties of $B_{k,\alpha}$, established in \cite{sibbel2024cantor} when $k=2$, that extend to the more general setting.
\begin{prp}\label{prop:Sibbel-Winter}
Let $k \geq 2$ and let $\alpha:X \rightarrow X$ be a minimal homeomorphism of the Cantor space. 
Let $\Phi_k:\K M_{k^\infty} \rightarrow \K M_{k^\infty}$ be the trace scaling automorphism. 
Let $E_1 = e_{11} \otimes 1_{M_{k^\infty}} \otimes 1_{C(X)} \in C(X, \K M_{k^\infty})$ be the constant function that evaluates to $e_{11} \otimes 1_{M_{k^\infty}}$ at all points of $X$. 
\begin{enumerate}[(i)]
    \item $B_{k,\alpha}$ is a simple, separable, nuclear, unital C$^*$-algebra; 
    \item $B_{k,\alpha} \cong E_1(C(X, \K M_{k^\infty}) \rtimes_{\Phi_k \otimes \alpha}\Z)E_1$;
    \item $\K \otimes B_{k,\alpha} \cong C(X, \K M_{k^\infty}) \rtimes_{\Phi_k \otimes \alpha}\Z$.
\end{enumerate} 
\end{prp}
\begin{proof}
    (i) Separability and unitality are clear. Nuclearity and simplicity were proven in \cite[Lemma 2]{sibbel2024cantor} in the case $k=2$ and the same proof works for all $k \geq 2$. 
    
    (ii) This was proven in \cite[Lemma 2]{sibbel2024cantor} in the case $k=2$. The same proof works for all $k \geq 2$.
    
    (iii) Let $A = C(X, \K M_{k^\infty}) \rtimes_{\Phi_k \otimes \alpha}\Z$. Let $E_i = e_{ii} \otimes 1_{M_{k^\infty}} \otimes 1_{C(X)} \in A$ for $i \in \N$. These projections are mutually orthogonal, mutually equivalent, and sum to the identity of $M(A)$ in the strict topology. By \cite[Theorem 2.1, Proof of (e) $\Rightarrow$ (a)]{hjelmborg1998stability}, $A \cong \K \otimes E_1AE_1$. The result now follows by (ii).
\end{proof}

In the proof of Theorem \ref{thm:Diagonal-in-Ok}, we intend to use the Kirchberg--Phillips classification theorem. We will therefore need to show that $B_{k,\alpha}$ is purely infinite and satisfies the universal coefficient theorem. The latter is an easy corollary of Proposition \ref{prop:Sibbel-Winter} and standard results in the literature.

\begin{prp}\label{prop:B-is-UCT}
Let $k \geq 2$ and let $\alpha:X \rightarrow X$ be a minimal homeomorphism of the Cantor space. Then $B_{k,\alpha}$ is in the bootstrap class and so satisfies the universal coefficient theorem.
\end{prp}
\begin{proof}
    The bootstrap class contains all AF algebras and is closed under stable isomorphism and crossed product by $\Z$; see for example \cite[Section 22.3.5]{Bla86}. The result now follows from Proposition \ref{prop:Sibbel-Winter}(iii).  
\end{proof}

It remains to show that $B_{k,\alpha}$ is purely infinite. For this, we will need a couple of preparatory lemmas. The first lemma is an application of a general result of R{\o}rdam and Sierakowski on pure infiniteness of crossed products (\cite[Theorem~3.3]{rordam2012purely}) to $\Z$-crossed products of stable AF algebras. 

\begin{lem}\label{lem:Rordam-Sierakowski}
Let $A$ be a stable AF algebra and $\beta:\Z \curvearrowright A$ be such that the induced action on $\Prim(A)$ is free.
Suppose $q \precsim p$ in $A \rtimes_\beta \Z$ for all projections $p,q\in A$ with $p \neq 0$. Then $A \rtimes_\beta \Z$ is purely infinite.
\end{lem}
\begin{proof}
    We check that conditions of \cite[Theorem 3.3]{rordam2012purely} are satisfied. Firstly, 
    since $A$ is AF, it is separable and has the ideal property. Secondly, as the group $\Z$ is discrete and amenable, the action is exact and the full and reduced crossed products coincide. Finally, the canonical surjection $\widehat{A} \twoheadrightarrow \mathrm{Prim}(A)$, which maps an irreducible representation to its kernel, intertwines the $\Z$-actions on $\widehat{A}$ and $\mathrm{Prim}(A)$ induced by $\beta$. 
    Since the $\Z$-action on $\mathrm{Prim}(A)$ induced by $\beta$ is assumed to be free, it follows that the $\Z$-action on $\widehat{A}$ induced by $\beta$ is also free.
    Therefore, the induced $\Z$-action on $\widehat{A}$  is essentially free as this is a weaker condition; see \cite[Definition~4.8]{renault1991ideal}. 

    Since $A$ is AF, it is separable and has real rank zero. 
    Therefore, the Cuntz semigroup $\Cu(A)$ is algebraic by \cite[Corollary 5]{CEI08}, i.e.\ every element is the supremum of a sequence of compact elements. 
    Moreover, as $A$ is stable and has stable rank one, every compact element of $\Cu(A)$ is represented by a projection in $A$; see for example \cite[Proposition 4.14]{GardellaPerera}. 
    Let $a \in A_+$ and $\epsilon > 0$. Using \cite[Proposition 3.11]{GardellaPerera}, there exists a projection $p \in A$ such that $(a-\epsilon)_+ \precsim p \precsim a$ in $A$.
    Since $A$ is stable, there is a projection $q \in A$ that is Cuntz equivalent to $p \oplus p$.
    Hence, by hypothesis, we have $p \oplus p \precsim p$ in $A \rtimes_\beta \Z$.\footnote{The case $p = 0$ is trivial.}
    It follows that $(a-\epsilon)_+ \oplus (a-\epsilon)_+ \precsim a$ in $A \rtimes_\beta \Z$. As this holds for all $\epsilon >0$, it follows that $a \oplus a \precsim a$ in $A \rtimes_\beta \Z$.
    Therefore, by \cite[Theorem 3.3]{rordam2012purely}, $A \rtimes_\beta \Z$ is purely infinite.
\end{proof}

We now prove that the main condition of Lemma \ref{lem:Rordam-Sierakowski} is satisfied by the crossed product of Proposition \ref{prop:Sibbel-Winter}(iii). 
\begin{lem}\label{lem:purely-infinite}
    Let $k \geq 2$. Let $\alpha:X \rightarrow X$ be a minimal homeomorphism on the Cantor space.
    Let $\Phi_k$ be the trace scaling automorphism on the stabilised UHF algebra $\K M_{k^\infty}$.
    Let $A = C(X, \K M_{k^\infty})$.
    Let $p,q \in A_+$ be projections with $p \neq 0$.
    Then $q \precsim p$ in $A \rtimes_{\Phi_k \otimes \alpha} \Z$. 
\end{lem}
\begin{proof}
   Since $p \neq 0$, there is a clopen set $Y \subseteq X$ such that $\Tr_{k^\infty}(p(x)) > 0$ for all $x \in Y$. Since $\bigcup_{i=1}^\infty \alpha^i(Y)$ is an open, non-empty, $\alpha$-invariant subspace, it must be all of $X$ by minimality. By compactness, $X = \bigcup_{i=1}^N \alpha^{i}(Y)$ for some $N \in \N$.  
    
    Since $Y$ is compact, there is $\epsilon > 0$ such that $\Tr_{k^\infty}(p(x)) > \epsilon$ for all $x \in Y$.  Let  $r_0 \in \K M_{k^\infty}$ be a projection with $0 < \Tr_{k^\infty}(r_0) < \tfrac{1}{N}\epsilon$. Then, by Proposition~\ref{prop:K-theory-functions-into-UHF}, we have $N[r_0 \otimes \mathcal{X}_Y]  \leq [p]$ in the Murray--von Neumann semigroup of $A$ and the Cuntz semigroup.

    Every projection $e \in A$ is unitarily equivalent to $(\Phi_k \otimes \alpha)(e)$ in the multiplier algebra of $A \rtimes_{\Phi_k \otimes \alpha} \Z$. So $[(\Phi_k \otimes \alpha)(e)] = [e]$ in $\Cu(A \rtimes_{\Phi_k \otimes \alpha} \Z)$. Since we have $(\Phi_k \otimes \alpha)(r_0 \otimes \mathcal{X}_Y) = \Phi_k(r_0) \otimes \mathcal{X}_{\alpha(Y)}$, it follows that
    \begin{equation}
        [p] \geq N[r_0 \otimes \mathcal{X}_Y] = \sum_{i=1}^N[\Phi_k^i(r_0) \otimes \mathcal{X}_{\alpha^i(Y)}]
    \end{equation}
    in $\Cu(A \rtimes_{\Phi_k \otimes \alpha} \Z)$.
    Let $r_1 \in \K M_{k^\infty}$ be a projection with $0 < \Tr_{k^\infty}(r_1) < \tfrac{1}{k^N}\Tr_{k^\infty}(r_0)$. 
    Then, by Proposition~\ref{prop:K-theory-functions-into-UHF}, we have $[r_1 \otimes \mathcal{X}_{\alpha^i(Y)}] \leq [\Phi_k^i(r_0) \otimes \mathcal{X}_{\alpha^i(Y)}]$ for $1 \leq i \leq N$. Hence, we have
    \begin{equation}
        [p] \geq \sum_{i=1}^N[r_1 \otimes \mathcal{X}_{\alpha^i(Y)}] \geq \left[r_1 \otimes \sum_{i=1}^N\mathcal{X}_{\alpha^i(Y)}\right] \geq [r_1 \otimes 1_{C(X)}]
    \end{equation}
    in $\Cu(A \rtimes_{\Phi_k \otimes \alpha} \Z)$, where the second inequality uses \cite[Lemma 3.2]{GardellaPerera}.
    
    Choose $m \in \N$ sufficiently large so that  $k^m\Tr_{k^\infty}(r_1) > \Tr_{k^\infty}(q(x))$ for all $x \in X$. Then, by Proposition~\ref{prop:K-theory-functions-into-UHF}, 
    $[q] \leq [\Phi_k ^{-m}(r_1) \otimes 1_{C(X)}]$ in $\Cu(A \rtimes_{\Phi_k \otimes \alpha} \Z)$. Therefore, we have
     \begin{equation}
        [r_1 \otimes 1_{C(X)}] = [(\Phi_k \otimes \alpha)^{-m}(r_1 \otimes 1_{C(X)})] = [\Phi_k ^{-m}(r_1) \otimes 1_{C(X)}] \geq [q].
    \end{equation}   
   Hence, $[q] \leq [p]$ in $\Cu(A \rtimes_{\Phi_k \otimes \alpha} \Z)$.
\end{proof}

We now use the previous two lemmas to deduce that $B_{k,\alpha}$ is purely infinite. 
\begin{prp}\label{prop:B-purely-infinite}
    Let $k \geq 2$ and $\alpha:X \rightarrow X$ be a minimal homeomorphism of the Cantor space.
    Then $B_{k, \alpha}$ is purely infinite.
\end{prp}
\begin{proof}
    Let $A = C(X, \K M_{k^\infty})$. Let $\Phi_k$ be the trace scaling automorphism on $\K M_{k^\infty}$.
    Since $\K M_{k^\infty}$ is simple, we have $X \cong \Prim(A)$ via the homeomorphism sending $x$ to the ideal of functions vanishing at $x$. 
    The induced action of $\Phi_k \otimes \alpha$ on $\Prim(A)$ now corresponds to the action of $\alpha$ on $X$. Since $\alpha:X \rightarrow X$ is a minimal homeomorphism the $\Z$-action that it induces is free. 
    Hence, by Lemma \ref{lem:Rordam-Sierakowski} and Lemma \ref{lem:purely-infinite}, $A \rtimes_{\Phi_k \otimes \alpha} \Z$ is purely infinite.
    By Proposition \ref{prop:Sibbel-Winter}, $B_{k, \alpha}$ is simple and stably isomorphic to $A \rtimes_{\Phi_k \otimes \alpha} \Z$.
    Hence,  $B_{k, \alpha}$ is purely infinite by \cite[Proposition 4.1.8(i)]{rordam2002classification}. 
\end{proof}

\section{K-theory of \texorpdfstring{$B_{k,\alpha}$}{B(k,alpha)}}\label{sec:K-theory}

In this section, we consider the K-theory of the Kirchberg algebras $B_{k,\alpha}$ constructed in the previous section. The basic idea is to use Proposition \ref{prop:Sibbel-Winter} to relate the K-theory of $B_{k,\alpha}$ to the K-theory of a $\Z$-crossed product, which can be analysed by means of the Pimsner--Voiculescu exact sequence (\cite{PVexact}).\footnote{Similar computations can be carried out using the 6-term exact sequences from \cite{Cuntz1981ktheoryII} or \cite{katsura2004class}.} For a general Cantor minimal system $\alpha$, we will show that $K_0(B_{k,\alpha}) \neq 0$ and $K_1(B_{k,\alpha}) = 0$. We will then specialise to the case of odometers, where we can compute $K_0(B_{k,\alpha})$ explicitly, and we deduce Theorem \ref{thm:Diagonal-in-Kirchberg}. 

We begin with the application of the Pimsner--Voiculescu exact sequence.
\begin{prp}\label{prop:general-K-theory}
    Let $k \geq 2$. Let $X$ be a zero-dimensional compact metric space. Let $\alpha:X \rightarrow X$ be a homeomorphism.
    Let $\Phi_k$ be the trace scaling automorphism on the stabilised UHF algebra $\K M_{k^\infty}$.
    Let $A = C(X, \K M_{k^\infty}) \rtimes_{\Phi_k \otimes \alpha} \Z$.
    The K-theory of $A$ is given by
    \begin{align}
        K_0(A) &\cong \frac{C_{lc}(X, \Z[k^{-1}])}{\Im\left(\id - \frac{1}{k}T\right)},\label{eqn:K0-crossed-product}\\
        K_1(A) &\cong 0,\label{eqn:K1-crossed-product}
    \end{align}
    where $T:C_{lc}(X, \Z[k^{-1}]) \rightarrow C_{lc}(X, \Z[k^{-1}])$  is the translation map $f \mapsto f \circ \alpha^{-1}$. Moreover, the isomorphism \eqref{eqn:K0-crossed-product} can be chosen to map the $K_0$-class of a projection $p \in C(X, \K M_{k^\infty})$ to the element $\Tr_{k^\infty} \circ p + \Im\left(\id - \frac{1}{k}T\right)$.
\end{prp}
\begin{proof}
    By Proposition \ref{prop:K-theory-functions-into-UHF}, $K_1(C(X, \K M_{k^\infty})) = 0$.
    Hence, the Pimsner--Voiculescu exact sequence has the form
   \begin{equation}\label{eqn:PV-6-term-exact-seq}
    \begin{tikzcd}[column sep=2cm]
    K_0(C(X, \K M_{k^\infty})) \arrow[r,"\id - (\Phi_k \otimes \alpha)_*"]  &K_0(C(X, \K M_{k^\infty})) \arrow[r,"\iota_*"] &K_0(A)\arrow[d]\\
    K_1(A) \arrow[u] &0 \arrow[l]  &0 \arrow[l],
    \end{tikzcd}
\end{equation} 
    where $\iota:C(X, \K M_{k^\infty}) \rightarrow A$ is the inclusion map. By exactness, we have 
    \begin{align}
	K_0(A)  &\cong \frac{K_0(C(X, \K M_{k^\infty}))}{\Im(\id - (\Phi_k \otimes \alpha)_*)},\label{eqn:PV-K0}\\
	K_1(A)  &\cong \Ker(\id - (\Phi_k \otimes \alpha)_*).\label{eqn:PV-K1}
    \end{align}
    By Proposition \ref{prop:K-theory-functions-into-UHF}, we may identify $K_0(C(X, \K M_{k^\infty}))$ with the group of locally constant functions $C_{lc}(X, \Z[k^{-1}])$ by composing with $\Tr_{k^\infty}$. Let $p \in C(X, \K M_{k^\infty})$ be a projection and $x \in X$. Then
    \begin{align}
        \Tr_{k^\infty}((\Phi_k \otimes \alpha)(p)(x)) = \Tr_{k^\infty}(\Phi_k(p(\alpha^{-1}(x)))) = \tfrac{1}{k}\Tr_{k^\infty}(p(\alpha^{-1}(x))).
    \end{align}
    Hence, $(\Phi_k \otimes \alpha)_* = \frac{1}{k}T$.
    This completes the proof of \eqref{eqn:K0-crossed-product} with the specified isomorphism. To prove \eqref{eqn:K1-crossed-product}, we need to show that $ \frac{1}{k}T$ has no non-zero fixed points.
    Suppose $f \in C_{lc}(X, \Z[k^{-1}]) \setminus \{0\}$ with $ \frac{1}{k}T(f) = f$. Then $f(x) = k^{-n}f(\alpha^{-n}(x))$ for all $x \in X$ and $n \in \N$. Hence, $f$ has infinite range. However, since $X$ is compact, a locally constant function on $X$ has a finite range. This contradiction completes the proof of \eqref{eqn:K1-crossed-product}.
\end{proof}

We now analyse the $K_0$-group further using tools from functional analysis.
\begin{lem}\label{lem:technical} 
    Let $k \geq 2$. Let $\alpha:X \rightarrow X$ be a homeomorphism of the compact Hausdorff space $X$.
    Let $T:C_{lc}(X,\Z[k^{-1}]) \rightarrow C_{lc}(X,\Z[k^{-1}])$ be the translation given by $f \mapsto f \circ \alpha^{-1}$.  
        Let $f \in C_{lc}(X, \D)$. Then  $f \in \Im\left(\id - \frac{1}{k}T\right)$ if and only if the function $g \in C(X)$ defined by
    \begin{equation}\label{eqn:g-Series}
        g(x) = \sum_{i=0}^\infty \frac{1}{k^i}f(\alpha^{-i}(x))
    \end{equation}
    satisfies $g \in C_{lc}(X, \Z[k^{-1}])$.
\end{lem}
\begin{proof}
    View $C_{lc}(X, \Z[k^{-1}]) \subseteq C(X)$. The translation $T$ extends to a bounded linear operator $\widehat{T}:C(X) \rightarrow C(X)$ with norm 1. Since $\|\tfrac{1}{k}\widehat{T}\| = \tfrac{1}{k} < 1$, we know that the operator $\id - \tfrac{1}{k}\widehat{T}$ is invertible with inverse given by the convergent series 
    \begin{equation}\label{eqn:Newton-Series}
        (\id - \tfrac{1}{k}\widehat{T})^{-1} = \sum_{i=0}^\infty \frac{1}{k^i}\widehat{T}^i.
    \end{equation}
    Suppose now that $f \in \Im(\id - \tfrac{1}{k}T)$. Then there exists $g \in C_{lc}(X, \D)$ with $f = (\id - \tfrac{1}{k}T)(g)$. Therefore, $f = (\id - \tfrac{1}{k}\widehat{T})(g)$. Hence, we have $g = (\id - \tfrac{1}{k}\widehat{T})^{-1}(f)$. Using \eqref{eqn:Newton-Series}, we obtain \eqref{eqn:g-Series}.
    Conversely, given $f \in C_{lc}(X,\D)$, the function $g \in C(X)$ given by \eqref{eqn:g-Series} is well defined as it satisfies $g = (\id - \tfrac{1}{k}\widehat{T})^{-1}(f)$, so we have $f = (\id - \tfrac{1}{k}\widehat{T})(g)$. If $g \in C_{lc}(X, \D)$, then  $f = (\id - \tfrac{1}{k}T)(g)$, so we have $f \in \Im(\id - \tfrac{1}{k}T)$. 
\end{proof}

Using this lemma, we can now show that the $K_0$-group of  $B_{k,\alpha}$ is never zero. 
\begin{prp}
    Let $k \geq 2$ and $\alpha:X \rightarrow X$ be a minimal homeomorphism of the Cantor space. Then $K_0(B_{k,\alpha}) \neq 0$ and $K_1(B_{k,\alpha}) = 0$. 
\end{prp}
\begin{proof}
    By Proposition \ref{prop:Sibbel-Winter}, $\K \otimes B_{k,\alpha} \cong C(X, \K M_{k^\infty}) \rtimes_{\Phi_k \otimes \alpha} \Z$. Since K-theory is stable, it follows from Proposition \ref{prop:general-K-theory} that $K_1(B_{k,\alpha}) = 0$ and $K_0(B_{k,\alpha})$ is isomorphic to the quotient of $C_{lc}(X, \Z[k^{-1}])$ by $\Im(\id-\tfrac{1}{k}T)$ where $T(f) = f \circ \alpha^{-1}$. 
    
    Let $Y \subseteq X$ be a non-trivial, proper clopen subset. 
    Let $\mathcal{X}_Y$ be the characteristic function of $Y$.
    Substituting $\mathcal{X}_Y$ into \eqref{eqn:g-Series}, we consider the function $g \in C(X)$ given by  
    \begin{equation}\label{eqn:chi-Series}
        g(x) = \sum_{i=0}^\infty \frac{1}{k^i} \mathcal{X}_Y(\alpha^{-i}(x))
    \end{equation}
    for all $x \in X$. 
    Since $\mathcal{X}_Y$ is $\{0,1\}$-valued, we may view the right hand side of \eqref{eqn:chi-Series} as a base-$k$ expansion of the real number $g(x)$.
    Suppose $g(x) \in \Z[k^{-1}]$ for some $x \in X$. Then either $\mathcal{X}_Y(\alpha^{-i}(x)) = 0$ for all sufficiently large $i \in \N$, or $k=2$ and $\mathcal{X}_Y(\alpha^{-i}(x)) = 1$ for all sufficiently large $i \in \N$. By minimality of $\alpha$, the backward orbit $\{\alpha^{-i}(x) \mid i \in \N\}$ is dense in $X$ and thus, it intersects infinitely often with $Y$ and with $X / Y$, so neither option can hold. Therefore, $\mathcal{X}_Y \not\in \Im(\id-\tfrac{1}{k}T)$ by Lemma \ref{lem:technical}. Hence, $K_0(B_{k,\alpha}) \neq 0$.
\end{proof}

This lets us answer the question posed by Mikael R{\o}rdam at the conference ``Noncommutativity in the North'' in Gothenburg in the Summer of 2024.
\begin{cor}
    Let $\alpha:X \rightarrow X$ be a minimal homeomorphism of the Cantor space. Then $B_{2,\alpha}$ is not isomorphic to $\O_2$.
\end{cor}
\begin{proof}
    We have $K_0(\O_2) = 0$ but $K_0(B_{2,\alpha}) \neq 0$.
\end{proof}

We now specialise to the case that $\alpha$ is an odometer. In this case, we can completely compute $K_0(B_{k,\alpha})$ by taking advantage of the fact that K-theory respects inductive limits.

We begin by fixing our notational conventions for odometers. Let $(n_i)_{i\in \N}$ be a strictly increasing sequence in $\N$ such that $n_i \mid n_{i+1}$.
We identify the Cantor space $X$ with the inverse limit of cyclic groups
\begin{equation}\label{eqn:projective-limit}
	X \rightarrow \cdots \xrightarrow{\rho_{5,4}} \Z_{n_4} \xrightarrow{\rho_{4,3}} \Z_{n_3} \xrightarrow{\rho_{3,2}} \Z_{n_2} \xrightarrow{\rho_{2,1}} \Z_{n_1},
\end{equation}
with the connecting group homomorphisms given by $\rho_{i+1,i}(x + n_{i+1}\Z) = x + {n_i}\Z$. The minimal homeomorphism $\alpha:X \rightarrow X$ given by $x \mapsto x+1$ is called the odometer corresponding to the sequence $(n_i)_{i\in \N}$. 

The action of $\alpha$ respects the inverse limit structure \eqref{eqn:projective-limit} and induces a cyclic shift on $\Z_{n_i}$ for each $i \in \N$. This gives rise to an inductive limit structure for the crossed product $C(X, \K M_{k^\infty}) \rtimes_{\Phi_k \otimes \alpha} \Z$.
The following proposition computes the $K_0$-group of the finite stages.
\begin{prp}\label{prop:K-theory-Zn-cross-product}
    Let $n \in \N$ and let $\alpha:\Z_n \rightarrow \Z_n$ be given by $x \mapsto x+1$.
    Let $k \geq 2$ and let $\Phi_k$ be the trace scaling automorphism on the stabilised UHF algebra $\K M_{k^\infty}$.
    Let $A = C(\Z_n, \K M_{k^\infty}) \rtimes_{\Phi_k \otimes \alpha} \Z$.
    Then there exists an isomorphism 
    \begin{equation}  \label{eqn:bar-psi}
        \bar{\psi}:K_0(A) \rightarrow \frac{\Z[k^{-1}]}{(k^n-1)\Z[k^{-1}]} 
   \end{equation}
    where 
    \begin{equation} \label{eqn:bar-psi-property} 
       \bar{\psi}([p]_{K_0(A)}) = \sum_{j=0}^{n-1} k^j\Tr_{k^\infty}(p(j)) + (k^n-1)\Z[k^{-1}]
   \end{equation}    
   for all projections $p \in C(\Z_n, \K M_{k^\infty}) \subseteq A$. 
\end{prp}
\begin{proof}
        Applying Proposition \ref{prop:general-K-theory}, we obtain
        \begin{align}
            K_0(A) &\cong \frac{C(\Z_n, \Z[k^{-1}])}{\Im\left(\id - \frac{1}{k}T\right)},\label{eqn:K0-Zn}
        \end{align}   
        where $Tf(x)=f(x-1)$ for all $f \in C(\Z_n, \Z[k^{-1}])$.  Moreover, the isomorphism \eqref{eqn:K0-Zn} can be taken to map the $K_0$-class of a projection $p \in C(\Z_n, \K M_{k^\infty})$ to $\Tr_{k^\infty} \circ p + \Im\left(\id - \tfrac{1}{k}T\right)$. We treat this isomorphism as an identification.
        
        Consider the group homomorphism
	\begin{equation}\label{eqn:K-theory-crossed-product-map}
        \begin{split}
        \psi:C(\Z_n, \D) &\rightarrow \displaystyle\frac{\D}{(k^{n}-1)\D}\\
            f &\mapsto \sum_{j=0}^{n-1} k^j f(j) + (k^{n}-1)\D,
        \end{split}
	\end{equation}
        where we view the index $j$ as an element of $\Z_n$.
	We compute that
	\begin{equation}
	\begin{split}
		\psi\left(\frac{1}{k}T(f)\right) &= \sum_{j=0}^{n-1} k^{j-1} f(j - 1) + (k^{n}-1)\D\\
		&= k^{-1}f(-1) + \sum_{j=1}^{n-1} k^{j-1} f(j - 1) + (k^{n}-1)\D\\
		&= k^{-1}f(-1) + \sum_{j=0}^{n-2} k^{j} f(j) + (k^{n}-1)\D\\
		&= k^{n-1}f(n-1) + \sum_{j=0}^{n-2} k^{j} f(j) + (k^{n}-1)\D\\
		&= \psi(f),
	\end{split}
	\end{equation}
	where we have used that $k^{-1} + (k^{n}-1)\D = k^{n-1} + (k^{n}-1)\D$ in the quotient ring. Therefore, $\psi$ descends to a well-defined group homomorphism $\bar{\psi}$ from the quotient of $C(\Z_n, \D)$ by $\Im\left(\id - \frac{1}{k}T\right)$. Identifying this quotient with $K_0(A)$, we see that \eqref{eqn:bar-psi-property} follows from \eqref{eqn:K-theory-crossed-product-map}.
    
    It remains to show that $\bar{\psi}$ is an isomorphism.
    By considering functions supported on $0 \in \Z_n$, it's clear from \eqref{eqn:K-theory-crossed-product-map} that $\psi$ is surjective. Hence, $\bar{\psi}$ is also surjective.  We next show that $\bar{\psi}$ is injective.
    Let $\delta_j$ for $j \in \Z_n$ be the standard $\D$-basis for $C(\Z_n, \D)$ given by functions that evaluate to 1 at the given element $j \in \Z_n$ and zero otherwise. Then $T(\delta_j) = \delta_{j+1}$ for each $j \in \Z_n$. 
    Writing $\tilde{\delta}_j = \delta_j + \Im\left(\id - \frac{1}{k}T\right)$ for $j \in \Z_n$, we have 
\begin{align}
	\tilde{\delta}_j &= k^{j}\tilde{\delta}_0 \text{ for } j = 0,1,\ldots, n-1, \text{ and } \label{eqn:delta-j}\\
	\tilde{\delta}_0 &= k^{n}\tilde{\delta}_0. \label{eqn:delta-0}
\end{align}
Let $x \in C(\Z_n, \D)$ and let $\tilde{x} = x + \Im\left(\id - \frac{1}{k}T\right)$. Writing $x = \sum_{j \in \Z_n} a_j\delta_j$ for $a_j \in \D$, we have $\tilde{x} = \left(\sum_{j \in \Z_n} k^{j}a_j\right)\tilde{\delta}_0$ by \eqref{eqn:delta-j}. 
Suppose $\bar{\psi}(\tilde{x}) = 0$. Then $\psi(x) = \sum_{j \in \Z_n} k^{j}a_j \in (k^{n}-1)\D$. Hence, $\tilde{x}  = (k^{n}-1)d\tilde{\delta}_0$ for some $d \in \D$. However, it follows from $\eqref{eqn:delta-0}$, that $(k^{n}-1)\tilde{\delta}_0 = 0$. Hence, $\tilde{x} = 0$, completing the proof that $\bar{\psi}$ is injective. Therefore, $\bar{\psi}$ is an isomorphism.
\end{proof}

We are now ready to compute the K-theory of $C(X, \K M_{k^\infty}) \rtimes_{\Phi_k \otimes \alpha} \Z$ when $\alpha$ is an odometer action.
\begin{prp}\label{prop:odometer-crossed-product}
    Let $(n_i)_{i=1}^\infty$ be a strictly increasing sequence in $\N$ with $n_i | n_{i+1}$ for all $i \in \N$.
    Let $\alpha:X \rightarrow X$ be the odometer corresponding to $(n_i)_{i=1}^\infty$.
    Let $k \geq 2$ and let $\Phi_k$ be the trace scaling automorphism on the stabilised UHF algebra $\K M_{k^\infty}$.
    Let $A = C(X, \K M_{k^\infty}) \rtimes_{\Phi_k \otimes \alpha} \Z$. 
    Then $K_1(A) = 0$ and $K_0(A)$ is given by the inductive limit
    \begin{equation}\label{eqn:inductive-limit-UHF}
	\Z_{m_1} \xrightarrow{\eta_{1,2}} \Z_{m_2}  \xrightarrow{\eta_{2,3}} \Z_{m_3} \xrightarrow{\eta_{3,4}} \cdots,
    \end{equation}
    where $m_i = k^{n_i}-1$ and the connecting map $\eta_{i,i+1}:\Z_{m_i} \rightarrow \Z_{m_{i+1}}$ is the group homomorphism $\eta_{i,i+1}(x) = \frac{m_{i+1}}{m_{i}}x$ for $x \in \Z_{m_i}$.
    Moreover, the $K_0$-class of the constant projection $e_{11} \otimes 1_{M_{k^\infty}} \otimes 1_{C(X)} \in C(X, \K M_{k^\infty}) \subseteq A$ corresponds to the sequence of elements $\frac{m_i}{k-1} + m_i\Z$ in the successive stages of inductive limit.
\end{prp}
\begin{proof}
    By Proposition \ref{prop:general-K-theory}, $K_1(A) = 0$.
    Using \eqref{eqn:projective-limit}, we can view $C(X, \K M_{k^\infty})$ as the inductive limit
\begin{equation}\label{eqn:inductive-limit}
	C(\Z_{n_1}, \K M_{k^\infty}) \xrightarrow{\iota_{1,2}} C(\Z_{n_2}, \K  M_{k^\infty}) \xrightarrow{\iota_{2,3}}  \cdots \rightarrow C(X, \K M_{2^\infty}),
\end{equation}
where the connecting maps are given by $\iota_{i,i+1}(f)(x+n_{i+1}\Z) = f(x+n_i\Z)$. The automorphism $\Phi_k \otimes \alpha$ preserves this inductive limit structure, i.e.\ we have an automorphism of inductive limits  
\begin{equation}\label{eqn:automorphism-inductive-limit}
    \begin{tikzcd}[column sep = small]
    C(\Z_{n_1},\K M_{k^\infty}) \arrow[r,"\iota_{1,2}"] \arrow[d,"\Phi_k \otimes \alpha_1"] &C(\Z_{n_2}, \K M_{k^\infty}) \arrow[r,"\iota_{2,3}"]\arrow[d,"\Phi_k \otimes \alpha_2"]  &\cdots \arrow[r] &C(X, \K M_{k^\infty})\arrow[d,"\Phi_k \otimes \alpha"]\\
    C(\Z_{n_1}, \K  M_{k^\infty}) \arrow[r,"\iota_{1,2}"] &C(\Z_{n_2}, \K  M_{k^\infty}) \arrow[r,"\iota_{2,3}"]  &\cdots \arrow[r] &C(X, \K  M_{k^\infty}),
    \end{tikzcd}
\end{equation}
where $(\Phi_k \otimes \alpha_i)(f)(x) = \Phi_k(f(x-1))$ for $i\in\N$, $f \in C(\Z_{n_i}, \K M_{k^\infty}) $ and $x \in \Z_{n_i}$. 

By \cite[Theorem 9.4.34]{GKPTbook},  $A$ is isomorphic to the inductive limit of the algebras $A_i = C(\Z_{n_i}, \K M_{k^\infty}) \rtimes_{\Phi_k \otimes \alpha_i} \Z$ with connecting maps given by $\iota_{i,i+1} \rtimes \id_\Z$.
Since K-theory respects inductive limits, it follows that $K_0(A)$ is the inductive limit of $K_0(A_i)$ with respect to the connecting maps $K_0(\iota_{i,i+1} \rtimes \id_\Z)$. 

By Propositions \ref{prop:K-theory-Zn-cross-product} and \ref{prop:K0-UHF-quotient}, we have $K_0(A_i) \cong \Z[k^{-1}]/m_i\Z[k^{-1}] \cong \Z_{m_i}$ where $m_i = k^{n_i}-1$. As $K_0(A_i)$ is cyclic, to understand the connecting map $K_0(\iota_{i,i+1} \rtimes \id_\Z)$ it suffices to consider the image of a generator for $K_0(A_i)$. 

For each $i \in \N$, let $\bar{\psi}_i:K_0(A_i) \rightarrow \Z[k^{-1}]/m_i\Z[k^{-1}]$ be the isomorphism given by \eqref{eqn:bar-psi}.
Fix $i \in \N$. Let $p \in C(\Z_{n_i}, \K M_{k^\infty})$ be a projection with $\Tr_{k^\infty}(p(0)) = 1$ and $p(x) = 0$ for $x \neq 0$. Then $\bar{\psi}_i(p) = 1$ by \eqref{eqn:bar-psi-property}, so $[p]$ generates $K_0(A_i)$. 
We have $\iota_{i,i+1}(p)(x) = p(0)$ for $x \in \{0,n_i,2n_i, \ldots, (n_{i+1}-n_i)\}$ and is zero otherwise. Hence,
\begin{equation}
\begin{split}
	\bar{\psi}_{i+1}(\iota_{i,i+1}(p)) &= \sum_{j=0}^{n_{i+1}-1} k^j\Tr_{k^\infty}(\iota_{i,i+1}(p)(j)) + m_{i+1}\Z[k^{-1}]\\
        &= (1+k^{n_i}+k^{2n_i} + \cdots + k^{n_{i+1}-n_i}) + m_{i+1}\Z[k^{-1}]\\
        &= \frac{k^{n_{i+1}}-1}{k^{n_{i}}-1} + m_{i+1}\Z[k^{-1}]\\
        &= \frac{m_{i+1}}{m_i} + m_{i+1}\Z[k^{-1}].
\end{split}
\end{equation}
Therefore, $K_0(\iota_{i,i+1} \rtimes \id_\Z) = \eta_{i,i+1}$. Finally, since $e_{11} \otimes 1_{M_{k^\infty}} \otimes 1_{C(\Z_{n_i})}$ has constant trace $1$, it follows from \eqref{eqn:bar-psi-property} that 
\begin{equation}
  \bar{\psi}_{i}([e_{11} \otimes 1_{M_{k^\infty}} \otimes 1_{C(\Z_{n_i})}]) = \sum_{j=0}^{n_{i}-1} k^j + m_i\Z[k^{-1}] = \frac{m_i}{k-1} + m_i\Z[k^{-1}]  
\end{equation}
by summing the geometric series.
\end{proof}

We now deduce Theorem \ref{thm:Diagonal-in-Kirchberg}.
\begin{proof}[Proof of Theorem \ref{thm:Diagonal-in-Kirchberg}]
    Let $B = B_{k,\alpha}$, where $\alpha:X \rightarrow X$ is the odometer corresponding to the sequence $(n_i)_{i=1}^\infty$.
    Then $B$ is a unital UCT Kirchberg algebra by Propositions \ref{prop:Sibbel-Winter}(i), \ref{prop:B-is-UCT} and \ref{prop:B-purely-infinite}.
    By Proposition \ref{prop:Sibbel-Winter}(iii), $\K \otimes B \cong C(X, \K M_{k^\infty}) \rtimes_{\Phi_k \otimes \alpha} \Z$.
    Hence, by Proposition \ref{prop:odometer-crossed-product}, $K_0(B) = \varinjlim \Z_{m_i}$ where $m_i = k^{n_i} -1$, and $K_1(B) = 0$, by Proposition \ref{prop:general-K-theory}.
\end{proof}

Increasing unions of cyclic groups are well understood. These groups can be identified with the dense subgroups of $\mathbb{T}$, where all elements have finite order, by choosing a suitable root of unity at each stage of the inductive limit. Moreover, this class of groups can be classified by considering the possible orders of elements. The following piece of elementary number theory shows that odometers with different supernatural numbers lead to C$^*$-algebras with non-isomorphic K-theory in Theorem \ref{thm:Diagonal-in-Kirchberg}. In particular, Theorem \ref{thm:Diagonal-in-Kirchberg} provides uncountably many examples of UCT Kirchberg algebras with C$^*$-diagonals.

\begin{prp}\label{prop:moonlighting-in-number-theory}
        Let $a,k \in \N$ with $k \geq 2$.  
	Suppose $p^s$ is a prime power such that $p^s | a$. 
        Then there exists a prime power $q^r$ such that $q^r | k^a-1$ and, for all $b \in \N$, $q^r \nmid k^b-1$ whenever $p^s \nmid b$. 
\end{prp}
\begin{proof}
	Since $p^s | a$, we have $k^{p^s} - 1 | k^{a}-1$.\footnote{This follows from the factorisation $x^{\alpha\beta}-1 = (x^\alpha-1)((x^\alpha)^{\beta-1} + (x^\alpha)^{\beta-2}  + \cdots + (x^\alpha) + 1)$.}
	By prime factorisation, there exists a prime power $q^r$ such that $q^r | k^{p^s} - 1$ but $q^r \nmid k^{p^{s-1}} - 1$. By transitivity of divisibility, $q^r | k^{a}-1$. 
	In the multiplicative group $\left(\Z_{q^r}\right)^\times$, we have $k^{p^s} = 1$ but $k^{p^{s-1}} \neq 1$. Hence, using that $p$ is prime, the order of $k$ in $\left(\Z_{q^r}\right)^\times$ is precisely $p^s$. Let $b \in \N$. Suppose that $q^r | k^{b}-1$. Then in $\left(\Z_{q^r}\right)^\times$, we have $k^b = 1$, so the order of $k$ must divide $b$, so $p^s | b$. The claim follows by taking the contrapositive.
\end{proof}

\begin{cor}
	 Let $k \geq 2$. Let $(n_i)_{i=1}^\infty$, $(n_i')_{i=1}^\infty$ be a strictly increasing sequences in $\N$ with $n_i | n_{i+1}$ and $n_i' | n_{i+1}'$ for all $i \in \N$. 
     Suppose there is a prime power dividing some $n_{i_0}$ that is not a factor of any number in the sequence $(n_i')_{i=1}^\infty$. Then
     \begin{equation}
         \varinjlim \frac{\Z}{(k^{n_i} - 1)\Z} \not\cong \varinjlim \frac{\Z}{(k^{n_i'} - 1)\Z}.
     \end{equation}
\end{cor}
\begin{proof}
	By Proposition \ref{prop:moonlighting-in-number-theory}, there exists a prime power $q^r$ such that $q^r | k^{n_{i_0}} - 1$  but $q^r \nmid k^{n_i'} - 1$ for all $i \in \N$.
        It follows that the group $\varinjlim \frac{\Z}{(k^{n_i} - 1)\Z} $ contains an element of order $q^r$ but $\varinjlim \frac{\Z}{(k^{n_i'} - 1)\Z}$ does not. 
\end{proof}

\section{\texorpdfstring{C$^*$}{C*}-diagonals in Cuntz algebras}\label{sec:main-theorem}

In this section, we prove that the tensor product of $B_{k,\alpha}$ and $M_\s$ is isomorphic to the Cuntz algebra $\O_k$ for a suitable choice of integer $k$, odometer action $\alpha$ and supernatural number $\s$. This is achieved using the Kirchberg--Phillips theorem and the K-theory computations of the preceding sections. We then deduce Theorem \ref{thm:Diagonal-in-Ok}.  

\begin{thm}\label{thm:Getting-Ok}
    Let $k \geq 2$ and set $n_i = k^{i-1}$ for all $i \in \N$.  Let $\alpha:X \rightarrow X$ be the odometer corresponding to $(n_i)_{i=1}^\infty$.
    Let $M_\s$ be the UHF algebra with supernatural number $\s = \prod_{p \in E} p^\infty$, where $E$ is the set of all prime numbers not dividing $k-1$.
    Then $B_{k,\alpha} \otimes M_\s \cong \O_{k}$.
\end{thm}
\begin{proof}
    By Propositions \ref{prop:Sibbel-Winter}(i), \ref{prop:B-is-UCT} and  \ref{prop:B-purely-infinite}, we see that $B_{k,\alpha}$ is a unital UCT Kirchberg algebra.
    The same is true for the tensor product $B_{k,\alpha} \otimes M_\s$.
    Indeed, the minimal tensor product preserves unitality, separability and nuclearity.
    Moreover, the bootstrap class is also closed under minimal tensor products (see for example \cite[Section 22.3.5(f)]{Bla86}). Finally, 
    the tensor product $B_{k,\alpha} \otimes M_\s$ is simple and purely infinite by \cite[Theorem 4.1.10]{rordam2002classification}.
    
    By the Kirchberg--Phillips theorem (see Theorem \ref{thm:KP}) and Cuntz's computation of the K-theory of $\O_k$ (see Theorem \ref{thm:Cuntz}), it therefore suffices to prove that $K_0(B_{k,\alpha} \otimes M_\s) \cong \Z_{k-1}$, with the $K_0$-class $[1_{B_{k,\alpha}} \otimes 1_{M_\s}]$ mapping to $1\in\Z_{k-1}$, and $K_1(B_{k,\alpha} \otimes M_\s) = 0$.

    By Proposition \ref{prop:Sibbel-Winter}(iii), $B_{k,\alpha}$ is stably isomorphic to $C(X, \K M_{k^\infty}) \rtimes_{\Phi_k \otimes \alpha} \Z$. Hence, by the K-theory computation of Proposition \ref{prop:odometer-crossed-product}, we see that $K_1(B_{k,\alpha}) = 0$ and $K_0(B_{k,\alpha})$ is given by the inductive limit \eqref{eqn:inductive-limit-UHF}. By Proposition \ref{prop:UHF-Kunneth}, we have $K_0(B_{k,\alpha} \otimes M_\s) \cong K_0(B_{k,\alpha}) \otimes K_0(M_\s)$ and $K_1(B_{k,\alpha} \otimes M_\s) = 0$. Hence, $K_0(B_{k,\alpha} \otimes M_\s)$ is given by the inductive limit
    \begin{equation} \label{eqn:inductive-limit-UHF-tensored}
	\Z_{m_1} \otimes  K_0(M_\s) \xrightarrow{\eta_{1,2} \otimes \id} \Z_{m_2} \otimes  K_0(M_\s) \xrightarrow{\eta_{2,3} \otimes \id} \Z_{m_3} \otimes  K_0(M_\s) \xrightarrow{\eta_{3,4} \otimes \id} \cdots,
    \end{equation}
    where $m_i = k^{n_i}-1$ and $\eta_{i,i+1}(x) = \frac{m_{i+1}}{m_i}x$; see for example \cite[Theorem 5.27]{Rotman09}. Since $K_0(M_\s)$ is a torsion free group, it is a flat $\Z$-module; see for example \cite[Corollary 3.50]{Rotman09}. By Proposition \ref{prop:odometer-crossed-product}, each map $\eta_{i,i+1}$ is injective. As $K_0(M_\s)$ is a flat $\Z$-module, each connecting map $\eta_{i,i+1} \otimes \id$ of \eqref{eqn:inductive-limit-UHF-tensored} is also injective.

    By Proposition \ref{prop:Sibbel-Winter}(ii), the unit $1_{B_{k,\alpha}} \in B_{k,\alpha}$ corresponds to the projection $e_{11} \otimes 1_{M_{k^\infty}} \otimes 1_{C(X)}$ in the crossed product $C(X, \K M_{k^\infty}) \rtimes_{\Phi_k \otimes \alpha} \Z$.
    By Proposition~\ref{prop:odometer-crossed-product}, this corresponds to the elements $\frac{m_i}{k-1}+m\Z_{m_i}$ in the successive stages of the inductive limit \eqref{eqn:inductive-limit-UHF}.
    Since $n_1 = k^0 = 1$, we have that $m_1 = k-1$ and $[1_{B_{k,\alpha}}]$ maps to $1 \in \Z_{k-1}$ at the first stage of the inductive limit \eqref{eqn:inductive-limit-UHF}.
    Hence, $[1_{B_{k,\alpha}} \otimes 1_{M_\s}]$ maps to the element $1 \otimes [1_{M_\s}]$ in $\Z_{k-1} \otimes  K_0(M_\s)$ at the first stage of the inductive limit \eqref{eqn:inductive-limit-UHF-tensored}  by Proposition \ref{prop:UHF-Kunneth}.
    
    We claim that our choice of $\s$ ensures that $\Z_{m_{i}} \otimes  K_0(M_\s) \cong \Z_{k-1}$ for all $i \in \N$.
    Fix $i \in \N$. Noting that $k-1$ divides $m_i$, set $m_i' = m_i / (k-1)$ for each $i \in \N$. 
    Let $p$ be a prime factor of $k-1$. Then, since $k \equiv 1$ (mod $p$) and $n_i = k^{i-1}$, we have
    \begin{equation}
        m_i' = \frac{k^{n_{i}}-1}{k-1} = 1 + k + k^2 + \cdots + k^{n_i-1} \equiv n_i \equiv 1 \, \mbox{(mod p)}.\\
    \end{equation}
    So, by our definition of $\s$, every prime factor of $m_i'$ must divide $\s$ with infinite multiplicity. Hence, $K_0(M_\s)$ is $m_i'$-divisible. 
    Therefore, we have 
    \begin{equation}
        \Z_{m_i} \otimes  K_0(M_\s) \cong \frac{K_0(M_\s)}{m_iK_0(M_\s)} \cong \frac{K_0(M_\s)}{(k-1)K_0(M_\s)} \cong \Z_{k-1},
    \end{equation}
    where the first isomorphism holds for all abelian groups, the second follows as $K_0(M_\s)$ is $m_i'$-divisible, and the third follows by Proposition \ref{prop:K0-UHF-quotient}. 

    Hence, the inductive limit \eqref{eqn:inductive-limit-UHF-tensored} is a stationary inductive limit with each term isomorphic to $\Z_{k-1}$. Since the connecting maps are injective, they are isomorphisms by the pigeonhole principle. Hence  $K_0(B_{k,\alpha} \otimes M_\s) \cong \Z_{k-1}$. Since $[1_{B_{k,\alpha}} \otimes 1_{M_\s}]$ maps to a generator at the first stage of the inductive limit, we may choose the isomorphism $K_0(B_{k,\alpha} \otimes M_\s) \cong \Z_{k-1}$ to send $[1_{B_{k,\alpha}} \otimes 1_{M_\s}]$ to $1 \in \Z_{k-1}$. The result now follows by the Kirchberg--Phillips theorem. 
\end{proof}

We can now complete the proof of Theorem \ref{thm:Diagonal-in-Ok}.

\begin{proof}[Proof of Theorem \ref{thm:Diagonal-in-Ok}]
    Let $k \in \N$ with $k \geq 2$. By Theorem \ref{thm:Getting-Ok}, there exists an odometer action $\alpha$ and a supernatural number $\s$ such that
    $B_{k,\alpha} \otimes M_{\s} \cong \O_k$. By Proposition \ref{prop:Diagonal-in-B},  $D_{k,\alpha}$ is a C$^*$-diagonal in $B_{k,\alpha}$ with Cantor spectrum. Let $D_\s$ be the canonical C$^*$-diagonal in $M_\s$. Then $D_{k,\alpha} \otimes D_\s$ is a C$^*$-diagonal in $B_{k,\alpha} \otimes M_{\s}$ by \cite[Lemma 5.1]{barlak2017cartan}. Since both $D_{k,\alpha}$ and $D_\s$ have Cantor spectrum, so does $D_{k,\alpha} \otimes D_\s$. Since $B_{k,\alpha} \otimes M_{\s} \cong \O_k$, the result follows.
\end{proof}   

\begin{rmk}
    It is natural to ask if our methods can be extended to cover $\O_\infty$. Indeed, Cuntz gave a crossed product description for $\K \otimes \O_\infty$ (\cite{cuntz1977simple}). Instead of a stabilised UHF algebra and a trace scaling automorphism, one now has a non-simple stable AF algebra $B$ with $K_0(B) \cong \bigoplus_{n \in \Z} \Z$ and the automorphism $\Phi$ acts on K-theory by a bilateral shift. Computations analogous to those in Section \ref{sec:K-theory} show that $K_0(C(X,B) \rtimes_{\Phi \otimes \alpha} \Z) = C(X, \Z)$ for any odometer action $\alpha$ on the Cantor space $X$. However, we are not aware of any tricks to obtain a Kirchberg algebra with the same K-theory as $\O_\infty$ from this.
\end{rmk}

\section{Groupoid models}\label{sec:groupoid-model}
In this section, we construct groupoid models for the C$^*$-algebras $B_{k,\alpha}$ defined in Section \ref{section:Construction} and deduce Theorem \ref{thm:Groupoid-model-for-Ok}. 
The main tool is Katsura's theory of topological graphs and their associated C$^*$-algebras (\cite{katsura2004class,katsura2006class3,katsura2008class}).
We would like to thank Ralf Meyer for drawing our attention to these results. 
We begin by briefly recalling the main definitions and constructions.

\begin{dfn}(\cite[Definition 2.1]{katsura2004class})\label{def:top-graph}
    A \emph{topological graph} $E= (E^0, E^1,d,r)$ consists of two locally compact spaces $E^0$ and $E^1$, and two maps $d,r: E^1 \to E^0$, where $d$ is locally homeomorphic and $r$ is continuous. 
\end{dfn}

Let $E = (E^0, E^1,d,r)$ be a topological graph. For simplicity, we assume $E^0$ and $E^1$ to be compact Hausdorff spaces and $r$ to be surjective. Specialising Katsura's construction from \cite{katsura2008class} to the compact setting, we equip $C(E^1)$ with the structure of a C$^*$-correspondence over $C(E^0)$ as follows: we define a $C(E^0)$-valued inner product on $C(E^1)$ by 
\begin{equation}
    \langle \xi, \eta \rangle (x) = \sum_{e \in d^{-1}(x)} \overline{\xi(e)} \eta(e)
\end{equation}
for $\xi, \eta \in C(E^1)$ and $x \in E^0$, noting that the sum is finite as $d$ is a local homeomorphism and $E^1$ is compact, and we define left and right actions of $C(E^0)$ by $(f \xi g) (e) = f(r(e)) \xi(e) g(d(e))$ for $f,g \in C(E^0), \xi \in C(E^1)$ and $e \in E^1$. 

The left action can be viewed as a $^*$-homomorphism $\pi_r: C(E^0) \rightarrow \mathcal{K}(C(E^1))$ into the C$^*$-algebra of compact operators on the right Hilbert $C(E^0)$-module $C(E^1)$, i.e.\  the closed span of the rank one adjointable operators $\theta_{\xi,\eta} \in \mathcal{L}(C(E^1))$ for $\xi,\eta \in C(E^1)$ defined by $\theta_{\xi,\eta} (\zeta) = \xi \langle \eta, \zeta \rangle$. 

\begin{dfn}\label{dfn:TopologicalGraphAlgebra} (\cite[Definition 1.3]{katsura2008class})
For a topological graph $E= (E^0, E^1,d,r)$ with $E^0$ and $E^1$ compact and $r$ surjective, the C$^*$-algebra $\mathcal{O}(E)$ is the universal C$^*$-algebra generated by the images of a $^*$-homomorphism $t^0:C(E^0) \to \mathcal{O}(E)$ and a linear map $t^1:C(E^1) \to \mathcal{O}(E)$ satisfying 
\begin{enumerate}
    \item[(i)]  $t^1(\xi)^*t^1(\eta) = t^0(\langle \xi, \eta \rangle) $ for $\xi, \eta \in C(E^1)$,
    \item[(ii)] $t^0(f)t^1(\xi) = t^1 (\pi_r(f)\xi)$ for $f \in C(E^0)$ and $\xi \in C(E^1)$,
    \item[(iii)] $t^0(f) = \varphi (\pi_r(f))$ for $f \in C(E^0)$,
\end{enumerate}
where $\varphi: \K(C(E^1)) \to \mathcal{O}(E)$ is the $^*$-homomorphism given by $\varphi(\theta_{\xi,\eta}) = t^1(\xi) t^1(\eta)^*$ for $\xi, \eta \in C(E^1)$.
\end{dfn}

We now show that the C$^*$-algebra $B_{k,\alpha}$, defined in \eqref{eqn:def-B-k-alpha}, is the C$^*$-algebra of a topological graph.

\begin{prp}\label{prop:Iso-to-graph-algebra}
    Let $k \geq 2$ and let $\alpha:X \rightarrow X$ be a minimal homeomorphism of the Cantor space.
    Then $B_{k,\alpha}$ is isomorphic to the C$^*$-algebra of the topological graph $E = (E^0, E^1, d,r)$ given by $E^0 = X$ and $E^1 = X \times \{1,\dots, k\}$
    with $d(x,i) = x$ and $r(x,i) = \alpha(x)$ for $(x,i) \in E^1$.
\end{prp}
\begin{proof}
    First, we note that $E$ satisfies Definition \ref{def:top-graph}. Next, we observe that $C(E^1) \cong C(X)^{\oplus k}$ as Hilbert $C(X)$-modules via the map $\xi \mapsto (\xi_1,...,\xi_k)$ where $\xi_i(x) = \xi(x,i)$ for all $(x,i) \in E^1$, and we treat this map as an identification.
    Under this identification, the left action is 
    $\pi_r(f) \xi = ((f\circ \alpha )\xi_1,\dots,(f\circ \alpha) \xi_k)$ for $f \in C(E^0)$ and $\xi = (\xi_1,\dots, \xi_k) \in C(E^1)$.

    Define $t^0 : C(E^0) \to B_{k,\alpha}$ and $t^1 : C(E^1) \to B_{k,\alpha}$ by $t^0(f) = 1_{\O_k} \otimes f$ and $t^1(\xi) = \sum_{i=1}^k S_i \otimes v\xi_i$
    for $f \in C(E^0)$ and $\xi = (\xi_1,\dots, \xi_k) \in C(E^1)$. Then, $t^0$ is a $^*$-homomorphism, $t^1$ is linear and the images of $t^0$ and $t^1$ generate $B_{k,\alpha}$. We check that conditions (i), (ii), and (iii) of Definition \ref{dfn:TopologicalGraphAlgebra} hold. 

    Let $\xi=(\xi_1,...,\xi_k),\eta=(\eta_1,...,\eta_k) \in C(E^1)$, then
\begin{equation}
        t^1(\xi)^*t^1(\eta) = \sum_{i = 1}^k\sum_{j = 1}^k S_i^*S_j \otimes \bar{\xi}_i\eta_j
        = \sum_{i = 1}^k 1_{\O_k} \otimes \bar{\xi}_i\eta_i
        = t^0(\langle \xi, \eta \rangle),
\end{equation}
    where we used that $S_i^*S_j = \delta_{ij}1_{\O_k}$. Hence, (i) holds.
    
   Let $f \in C(E^0)$ and $\xi=(\xi_1,...,\xi_k) \in C(E^1)$. Then 
\begin{equation}
         t^0(f)t^1(\xi) =  \sum_{i=1}^k S_i\otimes fv\xi_i = \sum_{i=1}^k S_i\otimes v(f\circ \alpha)\xi_1 = t^1 (\pi_r(f)\xi) 
\end{equation}
    using that $fv = vv^* fv = v (f\circ \alpha)$. Hence, (ii) holds.

    Let $f \in C(E^0)$. In order to show that (iii) is satisfied, we first note that
    \begin{equation}\label{eqn:LeftAction}
        \pi_r(f) = \theta_{f^1,e^1} +\dots + \theta_{f^k,e^k},
    \end{equation}
    where $e^i = (0,\dots, 0,1_{C(X)},0, \dots , 0)$ is the $i$-th standard basis element and $f^i = (0,\dots, 0,(f\circ \alpha),0, \dots , 0) = f e^i$. 
    Indeed, for any $\zeta=(\zeta_1,...,\zeta_k) \in C(E^1)$, we have 
    \begin{equation}
             \pi_r(f)(\xi) = ((f\circ \alpha) \zeta_1, \dots, (f\circ \alpha) \zeta_k) = \sum_{i=1}^k f^i \langle e^i, \zeta \rangle
    \end{equation}
    and \eqref{eqn:LeftAction} now follows.
    Using \eqref{eqn:LeftAction}, we see that
    \begin{equation}
       \varphi (\pi_r(f))= \sum_{i=1}^k \varphi(\theta_{f^i,e^i}) = \sum_{i=1}^k t^1(f^i) t^1(e^i)^*.
    \end{equation}
    By the definition of $t^1$, we obtain   
\begin{equation}
 t^1(f^i) t^1(e^i)^* = (S_i\otimes v(f\circ \alpha))  (S_i^* \otimes v^*) =  S_iS_i^*\otimes v(f\circ \alpha)v^*) = S_iS_i^*\otimes f
\end{equation}
for each $i\in\{1,\ldots,k\}$. Summing over $i$, we get
\begin{equation}
       \varphi (\pi_r(f)) = \sum_{i=1}^k t^1(f^i) t^1(e^i)^* 
       = \sum_{i=1}^k S_iS_i^*\otimes f 
       = 1_{\mathcal{O}_k} \otimes f 
       = t^0(f).
\end{equation}
    Hence, (iii) holds.    
    
    Therefore, (i), (ii) and (iii) from Definition \ref{dfn:TopologicalGraphAlgebra} are satisfied. By universality, there exists a surjective $^*$-homomorphism $\mathcal{O}(E) \to B_{k,\alpha}.$ 
    It remains to show that $\mathcal{O}(E)$ is simple, which then yields that the $^*$-homomorphism is in fact an isomorphism.
    By \cite[Corollary 8.13]{katsura2006class3}, the C$^*$-algebra of a topological graph with non-discrete set of vertices $E^0$ is simple if and only if the topological graph is minimal.
    Recall that a topological graph $E$ is called minimal if there exist no closed invariant sets other than $\emptyset$ or $E^0$ (see \cite[Definition 8.8]{katsura2006class3}). And since $\alpha$ is a minimal homeomorphism, this holds and completes the proof.
\end{proof}

We now deduce that $B_{k,\alpha}$ is a groupoid C$^*$-algebra using general results for topological graphs. The groupoids that arise here are Deaconu--Renault groupoids (see  \cite{deaconu1995groupoids, renault2000cuntz}). 

\begin{cor}\label{cor:Groupoid-model-for-B}
    Let $k \geq 2$ and let $\alpha:X \rightarrow X$ be a minimal homeomorphism of the Cantor space.
    There exists a second countable, principal, étale groupoid $\mathcal{G}_{k,\alpha}$ whose unit space is a Cantor space such that $C_r^*(\mathcal{G}_{k,\alpha}) \cong B_{k,\alpha}$.
\end{cor}
\begin{proof}
    Let $E=(E^0,E^1,d,r)$ be the topological graph from Proposition \ref{prop:Iso-to-graph-algebra} with $B_{k,\alpha} \cong \mathcal{O}(E)$.
    By \cite[Theorem 4]{katsura2009cuntz}, $\mathcal{O}(E)$ is in turn isomorphic to the C$^*$-algebra of the Deaconu--Renault groupoid 
    \begin{equation}
    \mathcal{G}_{k,\alpha}=\{(x,m-n,y) \mid x,y \in E^\infty, m,n \in \N \text{ and }\sigma^m(x) = \sigma^n(y)\},
    \end{equation}
    where $E^\infty = \{ (e_i)_{i \in \N} \in (E^1)^\N \mid d(e_i) = r(e_{i+1})\}$ is the infinite path space and $\sigma$ is the left shift on $E^\infty$. The groupoid structure on the Deaconu--Renault groupoid $\mathcal{G}_{k,\alpha}$ is given by $s(x,d,y) = (y,0,y)$, $r(x,d,y) = (x,0,x)$ with multiplication $(x,d_1,y)(y,d_2,z) = (x,d_1+d_2,z)$ and inverse $(x,d,y)^{-1} = (y,-d,x)$; 
the étale topology on $\mathcal{G}_{k,\alpha}$ has basic open sets 
\begin{equation} 
Z(U,m,n,V) = \{(x,m-n,y) : (x,y) \in U \times V \text{ with } \sigma^m(x) = \sigma^n(y)\},
\end{equation} where $m,n \in \N$ and $U$,$V$ are open sets in $E^\infty$ (see for example \cite[Theorem 1]{deaconu1995groupoids}). 
    
     Since $\alpha$ is a homeomorphism, $E^\infty$ is homeomorphic to $X \times \{1,\dots,k\}^\N$ with $\sigma$ conjugate under this homeomorphism to $\alpha \times \rho$, where $\rho$ is the left shift on $\{1,\dots,k\}^\N$. In particular, the unit space of $\mathcal{G}_{k,\alpha}$ is homeomorphic to $X \times \{1,\dots,k\}^\N$, which is a Cantor space. 
    Since $\alpha$ is minimal, $\alpha^n$ has no fixed points for all $n \in \Z$. Hence, we have $(\alpha \times \rho)^n(z) \neq (\alpha \times \rho)^m(z)$ for all $z \in X \times \{1,\dots,k\}^\N$ and all $n, m \in \N$ with $n \neq m$. Therefore, the groupoid $\mathcal{G}_{k,\alpha}$ is principal. As $E^\infty$ is a Cantor space, it is second countable. It follows from the description of the topology on $\mathcal{G}_{k,\alpha}$ above that $\mathcal{G}_{k,\alpha}$ is second countable.
\end{proof}

Combining Corollary \ref{cor:Groupoid-model-for-B} with known principal étale groupoid models for UHF algebras, we obtain a principal étale groupoid model for $\O_k$ when $2 \leq k < \infty$.
\begin{proof}[Proof of Theorem \ref{thm:Groupoid-model-for-Ok}]
    Let $k \geq 2$. Choose the same odometer $\alpha:X \rightarrow X$ and a supernatural number $\s$ as in Theorem  \ref{thm:Getting-Ok}. Then $\O_k \cong B_{k,\alpha} \otimes M_\s$. Let $\mathcal{G}_{k,\alpha}$ be the Deaconu--Renault groupoid with $B_{k,\alpha} \cong C^*_r(\mathcal{G}_{k,\alpha})$ from Corollary \ref{cor:Groupoid-model-for-B}. Let $\mathcal{R}_\s$ be the canonical principal étale groupoid for the UHF algebra $M_\s$ arising from an AF étale equivalence relation; see for example \cite[Theorem 3.1.15]{renault1980groupoid} or \cite{GPS04}.

    Set $\G_k = \mathcal{G}_{k,\alpha} \times \mathcal{R}_\s$. By \cite[Lemma 5.1]{barlak2017cartan}, the reduced C$^*$-algebra of a product of étale groupoids is isomorphic to the minimal tensor product of the individual reduced groupoid C$^*$-algebras. Hence, we have
    \begin{equation}
        C^*_r(\G_k) \cong C^*_r(\mathcal{G}_{k,\alpha}) \otimes C^*_r(\mathcal{R}_\s) \cong B_{k,\alpha} \otimes M_\s \cong \O_k .
    \end{equation}
    Since $\mathcal{G}_{k,\alpha}$ and $\mathcal{R}_\s$ are both second countable, principal, étale groupoids whose unit space is a Cantor space, the same is true for the product groupoid 
    $\G_k$.
\end{proof}

\begin{rmk}
    One can also apply the results of this section to the main result from \cite{sibbel2024cantor} to obtain $\O_2 \cong \bigotimes_{i\in\N} B_{2,\alpha} \cong C^*_r( \prod_{i\in\N} (\mathcal{G}_{2,\alpha},\mathcal{G}_{2,\alpha}^{(0)}))$, where the infinite product groupoid is defined by 
\begin{equation}\prod_{i\in\N} (\mathcal{G}_{2,\alpha},\mathcal{G}_{2,\alpha}^{(0)}) = \{ (\gamma_i)_{i\in\N} \in \prod_{i\in\N} \mathcal{G}_{2,\alpha}: \gamma_i \in \mathcal{G}_{2,\alpha}^{(0)} \text{ for almost all } i \in \N\};
\end{equation} see \cite[Lemma 5.2]{barlak2017cartan}.  It would be interesting to know whether these groupoids depend on the Cantor minimal system $\alpha$ and whether they are related to the groupoid model for $\O_k$ constructed above. 
\end{rmk}

\bigskip
\end{document}